\documentclass[12pt,reqno]{amsart}
\usepackage{amsmath,amsthm,amscd,amsfonts,amssymb,color}
\usepackage{cite}
\usepackage[mathscr]{eucal}
\usepackage[bookmarksnumbered,colorlinks,plainpages]{hyperref}
\newtheorem{theorem}{Theorem}[section]

\newtheorem{corollary}[theorem]{Corollary}
\newtheorem{open question}[theorem]{Open Question}
\theoremstyle{definition}
\newtheorem{definition}[theorem]{Definition}
\newtheorem{example}[theorem]{Example}
\theoremstyle{remark}
\newtheorem{remark}[theorem]{Remark}
\numberwithin{equation}{section}
\def\DJ{\leavevmode\setbox0=\hbox{D}\kern0pt\rlap
 {\kern.04em\raise.188\ht0\hbox{-}}D}

\begin{document}
\title[On contractive mappings in $b_v(s)$-metric spaces]{On contractive mappings in $b_v(s)$-metric spaces}

\author[P.\ Mondal, H.\ Garai, L.K. \ Dey]
{Pratikshan Mondal$^{1}$, Hiranmoy Garai$^{2}$, Lakshmi Kanta Dey$^{3}$.}

\address{{$^{1}$\,} Department of Mathematics,                                         					\newline \indent Durgapur Government College, 
                    Durgapur, India.}
                    \email{real.analysis77@gmail.com}
\address{{$^{2}$\,} Department of Mathematics,
                    \newline \indent National Institute of Technology
                    Durgapur, India.}
                    \email{hiran.garai24@gmail.com}
\address{{$^{3}$\,} Department of Mathematics,
                   \newline \indent National Institute of Technology
                    Durgapur, India.}
                    \email{lakshmikdey@yahoo.co.in}

\subjclass[2010]{$47$H$10$, $54$H$25$.}
\keywords{$b_v(s)$-metric space; complete space; sequentially compact space; contractive mapping}
\begin{abstract}

The major motives of this paper are to study different types of contractive mappings and also to answer an open question of Garai et al. [The contractive principle for mappings in $b_v(s)$-metric spaces, arXiv:1802.03136]. We first set up some fixed point results associated with two types of contractive mappings in $b_v(s)$-metric spaces and then we give an answer, in positive, to the open question. Most importantly, we characterize the completeness of a $b_v(s)$-metric space via fixed point property of a certain type of contractive mappings. Our results extend and generalized several important results in the literature.


\vskip 2mm

\textbf{Keywords:} $b_v(s)$-metric space; sequentially compact space; complete space; contractive mapping.

\end{abstract}

\maketitle

\section{Introduction and preliminaries}\label{sec:1}
Theory of fixed point is an interesting branch in analysis due its simplicity and applications. Many authors contributed to the theory with a numerous number of publications. This theory was originated by Banach \cite{B22} with an interesting and nice result, known as Banach contraction principle. The simplicity and usefulness of Banach contraction principle inspired many researchers to analyse it further. As a result, a number of generalizations and modifications emerge for this principle in different directions. One of these different directions is to change the underlying  metric space to different other abstract spaces. One of such abstract spaces is $b_v(s)$-metric space, which was introduced by Mitrovi\'c and Radenovi\'c \cite{MR} in $2017$. We first recall the definition of a $b_v(s)$-metric space. 
\begin{definition} {\bf(\cite[p. 3089, Definition 1.8]{MR}).}\label{d1}
Let $X$ be a non-empty set,  $v\in \mathbb{N}$ and $s\in [1,\infty)$. A function $\rho:X \times X \to \mathbb{R}$ is said to be a $b_{v}(s)$-metric if for all $x, y\in X$ and for all distinct points $u_1,u_2, \hdots, u_v \in X$, each of them different from $x$ and $y$, the following conditions hold:
\begin{itemize}
\item[{(i)}] $\rho(x,y) \geq 0$ and $\rho(x,y)=0$ $\iff$ $x=y$;
\item[{(ii)}] $\rho(x,y) = \rho(y,x)$;
\item[{(iii)}] $\rho(x,y)\leq s\left[\rho(x,u_1) + \rho(u_1,u_2) + \cdots + \rho(u_v,y)\right].$
\end{itemize} 
\end{definition} 
In this case, $(X,\rho)$ is said to be a $b_v(s)$-metric space. The notions of convergence, Cauchyness of a sequence, continuity of a mapping, completeness etc. can be seen in \cite{MR}. In the succeeding times, many authors contributed to the $b_v(s)$-metric fixed point theory with a number of fixed point results, see \cite{AMR18, DKMR20, MAKR19, AK18}.

Following all these theories, one can observe that these results  are only concerned with different types of contraction conditions. But it is known that standard metric fixed point theory is also enriched by different types of contractive conditions also, see \cite{E1, G, G1, F1, GDC, GDS}. So it is natural to focus on the fixed point results concerning different types of  contractive conditions in $b_v(s)$-metric spaces. Garai et al. \cite {GDM} focused in this direction at first. To do so, they first introduced the concepts of sequential and bounded compactness in the settings of $b_v(s)$-metric spaces, which are as follows.

\begin{definition}
Let $(X,\rho)$ be a $b_v(s)$-metric space. Then $X$ is said to be sequentially compact if for every sequence $\{u_n\}$ in $X$, there is a  subsequence of $\{u_n\}$ that converges to some point of $X$. Again a subset $A$ of $X$ is said to be sequentially compact if every for sequence $\{u_n\}$ in $A$, there is  a subsequence of $\{u_n\}$ that converges to some point of $A$.
\end{definition}

\begin{definition}
Let $(X,\rho)$ be a $b_v(s)$-metric space. Then $X$ is said to be boundedly compact if for every bounded sequence $\{u_n\}$ in $X$, there is  a subsequence of $\{u_n\}$, that converges to some point of $X$. Again a subset $A$ of $X$ is said to be boundedly compact if for every bounded sequence $\{u_n\}$ in $A$, there is a subsequence of $\{u_n\}$ that converges to some point of $A$.
\end{definition}
After this, Garai et al. proved some fixed point results related to contractive  mappings, i.e., a self-map $T$ defined on a $b_v(s)$-metric space $(X,\rho)$ satisfying $\rho(Tx,Ty)<\rho(x,y)$ for all $x,y\in X$ with $x\neq y$. They showed that a mapping satisfying contractive condition on a sequentially compact space acquires a  fixed point but not necessarily acquires a fixed point if the domain of the mapping is not sequentially compact but complete. So we need some additional condition(s) either on the underlying space or on the mapping so as to confirm the existence of fixed point. To find such an additional condition, Garai et al. obtained the following theorem.
\begin{theorem} {\bf(\cite[p. 11, Theorem 3.10]{GDM}).}\label{tt1}
Let $(X,\rho)$ be a complete $b_v(1)$-metric space, and let $T:X\to X$ be a contractive mapping.
Assume that for any $u\in X$ and for any $\varepsilon >0$, there exists $\delta >0$ such that $$\rho(T^nu,T^mu)< \varepsilon + \delta  \Rightarrow \rho(T^{n+1}u,T^{m+1}u)\leq \varepsilon$$ for any $n,m \in \mathbb{N}_0 $. Then $T$ acquires a unique fixed point.
\end{theorem}
We recognize that the additional assumption due to Theorem~\ref{tt1} does not deal with arbitrary $b_v(s)$-metric spaces, but deals with $b_v(1)$-metric spaces only. So it still remains interesting that what additional assumption(s) will work for arbitrary $b_v(s)$-metric spaces. Subsequently Garai et al. posed the following open problem.
\begin{open question}{\bf(\cite[\, p. 13, Open Question 3.11]{GDM}).}\label{tt2}
Let $(X,\rho)$ be a complete $b_v(s)$-metric space and let $T$ be a self-map on $X$ such that $$\rho(Tx,Ty)<\rho(x,y)$$ for all $x,y \in X$ with $x\neq y$. If $s>1$, then find out a weaker additional assumption on $T$ which will ensure that $T$ possesses a fixed point.
\end{open question}
In this paper, we deal with this open question. To do this, we first consider two  types of contractive mappings, viz., Reich type and \'Ciri\'c type. Then we establish some  results concerning these two types of contractive mappings in the settings of sequentially compact and complete $b_v(s)$-metric spaces. Utilizing these results, we give a positive answer to the open question \ref{tt2}. We further obtain a result from which we can characterize the completeness of  $b_v(s)$-metric spaces.  We also  provide some examples which support the results established in this paper and show that the conditions considered are not fictitious.

Throughout the paper, $\mathbb{N}_0$ stands for the set $\mathbb{N}\cup\{0\}$ and $\mathbb{R}^+$ stands for the set of all non-negative real numbers.

\section{Main Results}\label{sec2}
In this section, we first prove that a mapping defined on  sequentially compact $b_v(s)$-metric  spaces satisfying the Reich type contractive condition is a Picard operator.
\begin{theorem}\label{t1}
Let $(X, \rho)$ be a sequentially compact $b_v(s)$-metric  space. Let $T$ be a self-map on $X$ such that $T$ is orbitally continuous and
$$\rho(Tx, Ty)<a\rho(x, y)+b\rho(x, Tx)+c\rho(y, Ty)$$
for all $x, y\in X$ with $x\neq y$, where $a, b, c\in \mathbb{R}^+$ with $a+b+c=1$. Then $T$ acquires a unique fixed point $u$ (say), and for any $u \in X$, the Picard’s iterative sequence $\{T^n u\}$ converges to $u$. 
\end{theorem}
\begin{proof}
Let $u_0\in X$ be arbitrary but fixed. Define a sequence $\{u_n\}$ by $u_n=T^nu_0$ for all $n\in \mathbb{N}$. If $u_n=u_{n+1}$ for some $n\in \mathbb{N}$, then the result is obvious. So we now assume that $u_n\neq u_{n+1}$ for all $n\in \mathbb{N}$. Now note that if $c=1$, then $a=0=b$ and so we have $$\rho(u_{n+1}, u_{n+2})=\rho(Tu_n, Tu_{n+1})< c \rho(u_{n+1}, u_{n+2})=\rho(u_{n+1}, u_{n+2}),$$ which leads to a contradiction. Again if $a=1$, then $b=0=c$ and so the result follows from Theorem $3.1$ of \cite{GDM}. So for the rest of the proof, we  assume that $a,b,c<1$.

We set $s_n=\rho(u_n, u_{n+1})$ for all $n\in \mathbb{N}$. We claim that the sequence $\{s_n\}$ converges to $0$. To prove this, we first show that the sequence $\{s_n\}$ is strictly decreasing. We have
\begin{align*}
s_{n+1}&=\rho(u_{n+1}, u_{n+2})\\
&=\rho(Tu_n, Tu_{n+1})\\
&<a\rho(u_n, u_{n+1})+b\rho(u_n, u_{n+1})+c\rho(u_{n+1}, u_{n+2})\\
&=(a+b)s_n+cs_{n+1}\\
\Rightarrow (1-c)s_{n+1}&<(1-c)s_n\\
\Rightarrow s_{n+1}&<s_n.
\end{align*}

Therefore the sequence $\{s_n\}$ is strictly decreasing. Since $s_n\ge 0$ for all $n$, it follows that $s_n\to \alpha$ for some $\alpha\ge 0$. Now by sequential compactness of $X$, there exists a convergent subsequence, say $\{u_{n_k}\}$ of the sequence $\{u_n\}$. Let $u_{n_k}\to u\in X$ as $k\to \infty$. By the orbital continuity of $T$, we see that the subsequences  $\{u_{n_k+1}\}$ and $\{u_{n_k+2}\}$ converge to $Tu$ and $T^2u$ respectively. Then we have
\begin{align*}
\alpha&=\lim_{n\to \infty} \rho(u_n, u_{n+1})
=\lim_{k\to \infty} \rho(u_{n_k}, u_{n_k+1})
=\rho(u, Tu).
\end{align*}
Again we have
\begin{align*}
\alpha&=\lim_{n\to \infty} \rho(u_n, u_{n+1})
=\lim_{k\to \infty} \rho(u_{n_k+1}, u_{n_k+2})
=\rho(Tu, T^2u).
\end{align*}
We have already noted that $\alpha\ge 0$. If $\alpha>0$, then $u\ne Tu$ and then we have
$$\rho(Tu, T^2u)<a\rho(u, Tu)+b\rho(u, Tu)+c\rho(Tu, T^2u)$$
which implies that  $\rho(Tu, T^2u)<\rho(Tu, T^2u)$, and this leads to a contradiction. Then we have $u=Tu$ and consequently, $\alpha=0$, i.e., the sequence $\{s_n\}$ converges to $0$ and  $u$ is a fixed point of $T$.

We claim that $u$ is the only fixed point of $T$. If not, let $u_1\in X$ be a fixed point of $T$. Then we have
\begin{align*}
\rho(u, u_1)&=\rho(Tu, Tu_1)\\
&<a\rho(u, u_1)+b\rho(u, Tu)+c\rho(u_1, Tu_1)\\
&=a\rho(u, u_1)
\end{align*}
which implies that $a>1$, which is not possible here. Hence we must have $\rho(u, u_1)=0$, i.e, $u=u_1$. 

Finally we prove that $u_n=T^nu_0\to u$ as $n \to \infty$. If $u_{n_0}=u$ for some $n_0\in \mathbb{N}$, then
$u_n=u$ for all $n\ge n_0$ and so $u_n\to u$ in this case. 
Let us now define $t_n=\rho(u_n, u)$ for all $n\in \mathbb{N}$. Then,
\begin{align*}
0\le t_{n+1}&=\rho(u_{n+1}, u)\\
&=\rho(Tu_n, Tu)\\
&<a\rho(u_n, u)+b\rho(u_n, Tu_n)+c\rho(u, Tu)\\
&=a\rho(Tu_{n-1}, Tu)+b\rho(u_n, u_{n+1})\\
&<a\{a\rho(u_{n-1}, u)+b\rho(u_{n-1}, u_n)\}+bs_n\\
&=a^2t_{n-1}+abs_{n-1}+bs_n\\
&\cdots\\
&\cdots\\
&<a^{n+1}\rho(u_0,u)+a^nb\rho(u_0, u_1)+a^{n-1}bs_1+\cdots+abs_{n-1}+bs_n\\
&=a^{n+1}\rho(u_0,u)+b\{a^n\rho(u_0, u_1)+a^{n-1}s_1+\cdots+as_{n-1}+s_n\}\\
&\to 0 \mbox{ as }n\to \infty.
\end{align*}

Thus  the proof is done.
\end{proof}
As special cases of the above theorem, we have the following two existing important results:
\begin{corollary} \label{c1}{\bf\cite[\, p. 74, Theorem 1]{E1}.}
Let $(X,\rho)$ be a compact metric space and $T$ be a self-map on $X$ such that 
$$
\rho(Tx,Ty)<\rho(x,y)
$$ 
for all $x,y \in X$ with $x\neq y$. Then $T$ has a unique fixed point.
\end{corollary}

\begin{corollary}\label{c2}{\bf\cite[\, p. 2147, Theorem 2.2]{G1}.}
Let $(X,\rho)$ be a compact metric space and $T$ be a continuous self-map on $X$ such that 
$$
\rho(Tx,Ty)<\frac{1}{2}\Big\{\rho(x,Tx)+\rho(y,Ty)\Big\}
$$
 for all $x,y \in X$ with $x\neq y$. Then $T$ has a unique fixed point.
\end{corollary}
Next, we prove an analogous result of Theorem~\ref{t1} in the structure of complete $b_v(s)$-metric spaces. Before proving this result, let us consider the following example:
\begin{example}\label{e2}
Let $X=\left\{\frac{1}{n}:n\ge 2\right\}$. Define a function $\rho:X\times X\to \mathbb{R}$ by
$$ \rho\left(\frac{1}{m}, \frac{1}{n}\right)=\left\{%
\begin{array}{ll}
    |m-n| &\hbox{if $|m-n|\neq 1$}\\
    \frac{1}{2} & \hbox{if $|m-n|=1$}.
\end{array}%
\right.$$ Then $(X, \rho)$ is a complete $b_3(3)$-metric space which is not sequentially compact.

We now define an operator $T:X\to X$ as
$$ T\left(\frac{1}{n}\right)=\left\{%
\begin{array}{ll}
    \frac{1}{2} &\hbox{if $n>2$}\\
    \frac{1}{4} & \hbox{if $n=2$}.
\end{array}%
\right.$$ Note that $T$ is not a contraction, as $\rho\left(\frac{1}{2}, \frac{1}{3}\right)=\frac{1}{2}<\rho\left(T\left(\frac{1}{2}\right), T\left(\frac{1}{3}\right)\right)=2$. It is also easy to verify that $T$ satisfy
$$\rho(Tx, Ty)<\frac{1}{3}\rho(x, y)+\frac{1}{3}\rho(x, Tx)+\frac{1}{3}\rho(y, Ty)$$
for all $x, y\in X$ with $x\neq y$. However $T$ admits no fixed point in $X$. 
\end{example}
\begin{remark}
The above example shows that the condition $\rho(Tx, Ty)<\frac{1}{3}\rho(x, y)+\frac{1}{3}\rho(x, Tx)+\frac{1}{3}\rho(y, Ty)$ is not sufficient to guarantee the existence of fixed point of a mapping in the setting of a complete $b_v(s)$-metric space. Thus we need to consider some additional condition to assure the existence of a fixed point, which is reflected in the following theorem.
\end{remark}
\begin{theorem}\label{t3}
Let $(X, \rho)$ be a complete $b_v(s)$-metric  space. Let $T:X\to X$ be a mapping satisfying the contractive condition of Theorem~ \ref{t1}.  Furthermore, assume that for any $x\in X$ and  for any  $\varepsilon>0$, there exists a $\delta>0$ and an $N\in \mathbb{N}$ such that for $n,m\in \mathbb{N}$ with $n,m \geq N,$
$$\rho(T^nx, T^mx)<s^2\varepsilon+\delta\Longrightarrow \rho(T^{n+1}x, T^{m+1}x)\le \varepsilon.$$
 Then all the conclusions of Theorem \ref{t1} hold good.
\end{theorem}
\begin{proof}
For arbitrary $u_0\in X$, consider the sequence $\{u_n\}$ defined by $u_n=T^n{u_0}$ for each $n\in \mathbb{N}$. In case $u_n=u_{n+1}$ for some $n\in \mathbb{N}$, then $u_n$ is the unique fixed point of $T$ and we are done.

Let us suppose that $u_n\ne u_{n+1}$ for all $n\in \mathbb{N}$. Set $s_n=\rho(u_n, u_{n+1})$ for all $n\in \mathbb{N}$. Then proceeding similarly, as in Theorem~\ref{t1}, we find that the sequence $\{s_n\}$ is strictly decreasing. 

Since $s_n\ge 0$ for all $n$, it follows that $s_n\to \alpha$ for some $\alpha\ge 0$. If $\alpha>0$, then by given condition there exists a $\delta'>0$  and an $N_1\in \mathbb{N}$ such that 
$$\rho(u_n, u_{n+1})<s^2\alpha+\delta'\Longrightarrow \rho(u_{n+1}, u_{n+2})\le \alpha$$
for all $n\geq N_1$. 
By definition of $\alpha$, for this $\delta'>0$, there exists a sufficiently large $n\in \mathbb{N}$ such that
$$\rho(u_n, u_{n+1})<\alpha+\delta'\leq s^2\alpha+\delta'.$$
Therefore, $\rho(u_{n+1}, u_{n+2})\le \alpha$ and this leads to a contradiction. Hence we must have $\alpha=0$ i.e., $\displaystyle\mathop{\lim_{n\to \infty} \rho(u_n, u_{n+1})=0}$.

Next, we show that $\{u_n\}$ is a Cauchy sequence. Let $\varepsilon>0$ be arbitrary. Then we get a $\delta>0$ and an $N_2\in \mathbb{N}$ such that 
$$\rho(T^ix, T^jx)<s^2\varepsilon+\delta\Longrightarrow \rho(T^{i+1}x, T^{j+1}x)\le \varepsilon$$
for all $i,j\geq N_2$.

Without loss of generality, we can assume that $\delta\le \varepsilon$. Since $\displaystyle\mathop{\lim_{n\to \infty} \rho(u_n, u_{n+1})=0}$, there exists an $N_3\in \mathbb{N}$ such that
$$\rho(u_n, u_{n+1})<\frac{\delta}{4(v+1)s^2}$$
for all $n\ge N_3$.

Let $n\in \mathbb{N}$ with $n\ge \max\{N_2,N_3\}+1$ be arbitrary. We now show by method of induction that
$$\rho(u_n, u_{n+k})\le \varepsilon$$
for all $k\in \mathbb{N}$.

Clearly, the result is true for $k=1$. Let the result be true for $k=1,2,\cdots, m$.

Case I: Let us first assume that $m>v$. Then
\begin{align*}
&\rho(u_{n-1}, u_{n+m})\\
&\le s\{\rho(u_{n-1}, u_n)+\rho(u_n, u_{n+1})+\cdots+\rho(u_{n+v-2}, u_{n+v-1})+\rho(u_{n+v-1}, u_{n+m})\}\\
&<s\Big\{\rho(u_{n-1}, u_n)+\rho(u_n, u_{n+1})+\cdots+\rho(u_{n+v-2}, u_{n+v-1})+a\rho(u_{n+v-2}, u_{n+m-1})\\
& +b\rho(u_{n+v-2}, u_{n+v-1})+c\rho(u_{n+m-1}, u_{n+m})\}\Big\}\\
&<s\left\{\frac{\delta}{4(v+1)s^2}+\cdots+\frac{\delta}{4(v+1)s^2}+(b+c)\frac{\delta}{4(v+1)s^2}+a\rho(u_{n+v-2}, u_{n+m-1})\right\}\\
&<\frac{\delta}{4s}+\frac{\delta}{4(v+1)s}+as\Big\{a\rho(u_{n+v-3}, u_{n+m-2})+b\rho(u_{n+v-3}, u_{n+v-2})\\
&+c\rho(u_{n+m-2}, u_{n+m-1})\Big\}\\
&<\frac{\delta}{4s}+2\frac{\delta}{4(v+1)s}+a^2s\rho(u_{n+v-3}, u_{n+m-2})\\
&\cdots\\
&<\frac{\delta}{4s}+v\frac{\delta}{4(v+1)s}+a^vs\rho(u_{n}, u_{n+m-v})\\
&<\frac{\delta}{2s}+s\varepsilon.
\end{align*} 

Case II: We now assume that $m<v$. Then
\begin{align*}
&\rho(u_{n-1}, u_{n+m})\\
&<s\{\rho(u_{n+m}, u_{n+m+1})+\rho(u_{n+m+1}, u_{n+m+2})+\cdots+\rho(u_{n+m+v-1}, u_{n+m+v})\\& +\rho(u_{n+m+v}, u_{n-1})\}.
\end{align*}
By Case I, we can conclude that
$$\rho(u_{n+m+v}, u_{n-1})<s\varepsilon+\frac{\delta}{2s}.$$
Therefore, we get
$$\rho(u_{n-1}, u_{n+m})<s\left\{\frac{\delta}{4(v+1)s^2}+\frac{\delta}{4(v+1)s^2}+\cdots+\frac{\delta}{4(v+1)s^2}+s\varepsilon+\frac{\delta}{2s}\right\}<\delta+s^2\varepsilon.$$

Case III: Let us finally consider $m=v$. In this case
\begin{align*}
&\rho(u_{n-1}, u_{n+m})\\
&\le s\{\rho(u_{n-1}, u_n)+\rho(u_n, u_{n+1})+\cdots+\rho(u_{n+v-2}, u_{n+v-1})+\rho(u_{n+v-1}, u_{n+v})\}\\
&<s\left\{\frac{\delta}{4(v+1)s^2}+\frac{\delta}{4(v+1)s^2}+\cdots+\frac{\delta}{4(v+1)s^2}\right\}\\
&<\frac{\delta}{2s}<s\varepsilon+\frac{\delta}{2s}.
\end{align*}
Thus, by combining all three cases, we find that
$$\rho(u_{n-1}, u_{n+m})<s^2\varepsilon+\delta.$$
Then by hypothesis, we get
$$\rho(u_n, u_{n+m+1})\le \varepsilon$$
which shows that the result is true for $k=m+1$. Therefore, by method of induction, we get
$$\rho(u_n, u_{n+k})\le \varepsilon$$
for all $n\ge \max\{N_2,N_3\}+1$ and for all $k\in \mathbb{N}$. Hence $\{u_n\}$ is a Cauchy sequence in $X$ and by completeness of $X$, we find an element $z\in X$ such that $u_n\to z$ as $n\to \infty$. That $z$ is the unique fixed point of $T$ and the sequence $\{T^nu_0\}$ converges to $z$ follow along the same line of proof of Theorem~\ref{t1}. 
The proof is complete.
\end{proof}
The above theorem extends the following theorem due to Suzuki.
\begin{corollary}\label{c5}{\bf\cite[\, p. 2362, Theorem 5]{S2}.}
Let $(X,\rho)$ be a complete metric space and $T$ be a contractive mapping on $T$. Further, assume that for any $x\in X$ and  for any  $\varepsilon>0$, there exists a $\delta>0$ and an $N\in \mathbb{N}$ such that for $n,m\in \mathbb{N}$ with $n,m \geq N,$
$$\rho(T^nx, T^mx)<\varepsilon+\delta\Longrightarrow \rho(T^{n+1}x, T^{m+1}x)\le \varepsilon.$$
 Then all the conclusions of Theorem \ref{t1} hold good.
\end{corollary}
Let us now consider the following example, which is in support of the above theorem.
\begin{example}\label{e4}
Let $X=[0, \infty)$. Define $\rho:X\times X\to \mathbb{R}$ by
$$ \rho(x, y)=\left\{%
\begin{array}{ll}
    0 &\hbox{if $x=y$}\\
    2x & \hbox{if $x\ne 0, y=0$}\\
    2y &\hbox{if $x=0, y\ne 0$}\\
    4(x+y)+1 & \hbox{if $x\ne 0, y\ne 0$}.
\end{array}%
\right.$$
Then $(X, \rho)$ is a complete $b_2(2)$-metric space which is not sequentially compact.

Let us define $T:X\to X$, as
$$ Tx=\left\{%
\begin{array}{ll}
    0 &\hbox{if $x\in [0, 1)$}\\
    \frac{x+1}{4} & \hbox{if $x\ge 1$}.
\end{array}%
\right.$$
Let $x, y\in X$ with $x\ne y$. We now consider the following three cases:

Case-I: Let $x, y\in [0, 1)$. Then $Tx=0=Ty$. So
$$\rho(Tx, Ty)<\frac{1}{3}\rho(x, y)+\frac{1}{3}\rho(x, Tx)+\frac{1}{3}\rho(y, Ty).$$

Case-II: Let $x, y\ge 1$. Then $Tx=\frac{x+1}{4}, Ty=\frac{y+1}{4}$. 

Now, $\rho(Tx, Ty)=4\left(\frac{x+1}{4}+\frac{y+1}{4}\right)+1=x+y+3$.

Also, $\rho(x, y)=4x+4y+1, \rho(x, Tx)=\rho\left(x, \frac{x+1}{4}\right)=5x+2, \rho(y, Ty)=5y+2$. 

Then,
\begin{align*}
&\rho(Tx, Ty)-\left\{\frac{1}{3}\rho(x, y)+\frac{1}{3}\rho(x, Tx)+\frac{1}{3}\rho(y, Ty)\right\}\\
&=(x+y+3)-\frac{1}{3}\{4x+4y+1+5x+2+5y+2\}\\
&=-2x-2y+\frac{4}{3}.
\end{align*}
Since $x\ge 1, y\ge 1$, we have 
$$\rho(Tx, Ty)-\left\{\frac{1}{3}\rho(x, y)+\frac{1}{3}\rho(x, Tx)+\frac{1}{3}\rho(y, Ty)\right\}\le -2-2+\frac{4}{3}<0$$
which implies that $\rho(Tx, Ty)<\frac{1}{3}\rho(x, y)+\frac{1}{3}\rho(x, Tx)+\frac{1}{3}\rho(y, Ty).$

Case-III: Let $x\in [0, 1)$ and $y\ge 1$. Then $Tx=0$ and $Ty=\frac{y+1}{4}$ and so $\rho(Tx, Ty)=\frac{y+1}{2}$.

Now, $$ \rho(x, Tx)=\left\{%
\begin{array}{ll}
    0 &\hbox{if $x=0$}\\
    2x & \hbox{if $x\ne 0$}
\end{array}%
\right. and \ \ \ \rho(y, Ty)=5y+2.$$
Therefore,
$$\rho(Tx, Ty)-\frac{1}{3} \rho(y, Ty)=\frac{y+1}{2}-\frac{5y+2}{3}=\frac{-7y-1}{6}<0$$
and this implies that
$$\rho(Tx, Ty)<\frac{1}{3}\rho(x, y)+\frac{1}{3}\rho(x, Tx)+\frac{1}{3}\rho(y, Ty).$$
Now, let $x\in X$ and let $\varepsilon>0$ be arbitrary.

Case-I: Let $x\in [0, 1)$. Here we choose $\delta=\varepsilon$ and $N=1$. Then clearly,
$$\rho(T^ix, T^jx)<s^2\varepsilon+\delta\Longrightarrow \rho(T^{i+1}x, d^{j+1}x)=0\le \varepsilon.$$

Case-II: Let $x\ge 1$. Then $T^nx=\frac{x+a_n}{4^n}$ where $a_1=1, a_n=a_{n-1}+4^{n-1}$. Then there exists an $N\in \mathbb{N}$ such that $T^Nx<1$.

Let $\varepsilon>0$ be arbitrary. Choose $\delta=\varepsilon$. Then $T^i(T^Nx)=0$ for all $i\in \mathbb{N}$. Therefore,
$$\rho(T^ix, T^jx)<s^2\varepsilon+\delta\Longrightarrow \rho(T^{i+1}x, T^{j+1}x)=0\le \varepsilon$$ 
for $i, j\ge N+1$.

Thus all the conditions of Theorem~\ref{t3} are satisfied. Note that $0$ is the unique fixed point of $T$.
\end{example}
Next, we prove an analogous version of Theorem~\ref{t1} by changing the contractive condition.
\begin{theorem}\label{t5}
Let $(X, \rho)$ be a sequentially compact $b_v(s)$-metric  space. Let $T$ be a self-map on $X$ such that $T$ is orbitally continuous and
$$\rho(Tx, Ty)<\max\big\{\rho(x, y), \rho(x, Tx), \rho(y, Ty)\big\}$$
for all $x, y\in X$ with $x\neq y$. Then all the conclusions of Theorem \ref{t1} hold.
\end{theorem}
\begin{proof}
Let $u_0\in X$ be arbitrary but fixed. Define a sequence $\{u_n\}$ by $u_n=T^nu_0$ for all $n\in \mathbb{N}$. If $u_n=u_{n+1}$ for some $n\in \mathbb{N}$, then the result is obvious. 

We set $s_n=\rho(u_n, u_{n+1})$ for all $n\in \mathbb{N}$. We claim that the sequence $\{s_n\}$ converges to $0$. To prove this, we first show that the sequence $\{s_n\}$ is strictly decreasing. We have
\begin{align*}
s_{n+1}&=\rho(u_{n+1}, u_{n+2})\\
&=\rho(Tu_n, Tu_{n+1})\\
&<\max\{\rho(u_n, u_{n+1}), \rho(u_n, u_{n+1}), \rho(u_{n+1}, u_{n+2})\}\\
&=\max\big\{\rho(u_n, u_{n+1}), \rho(u_{n+1}, u_{n+2})\big\}\\
&=\rho(u_n, u_{n+1})=s_n\\
\Rightarrow s_{n+1}&<s_n.
\end{align*}

Therefore the sequence $\{s_n\}$ is strictly decreasing. Since $s_n\ge 0$ for all $n$, it follows that $s_n\to \alpha$ for some $\alpha\ge 0$. Now sequential compactness of $X$ yields that there is a convergent subsequence, say $\{s_{n_k}\}$ of the sequence $\{s_n\}$. Let $s_{n_k}\to u\in X$ as $k\to \infty$. By the orbital continuity of $T$, we can show that $\alpha=\rho(u, Tu)=\rho(Tu, T^2u).$
We have already noted that $\alpha\ge 0$. If $\alpha>0$, then $u\ne Tu$ and then we have
$$\rho(Tu, T^2u)<\max\{\rho(u, Tu), \rho(u, Tu), \rho(Tu, T^2u)\}$$
which implies that  $\alpha<\alpha$, which is a contradiction. So we must have $\alpha=0$ and therefore $u=Tu$, i.e., the sequence $\{u_n\}$ converges to $0$ and $u$ is a fixed point of $T$.

For uniqueness of the fixed point, let $u_1\in X$ be another fixed point of $T$. We claim that $\rho(u, u_1)=0$. If not, then
\begin{align*}
\rho(u, u_1)&=\rho(Tu, Tu_1)\\
&<\max\{\rho(u, u_1), \rho(u, Tu), \rho(u_1, Tu_1)\}\\
&=\rho(u, u_1),
\end{align*}
a contradiction. This proves the uniqueness of $u$.

Finally we prove that $u_n=T^nu_0\to u$ as $n \to \infty$. If $u_{n_0}=u$ for some $n_0\in \mathbb{N}$, then
$u_n=u$ for all $n\ge n_0$ and hence $u_n\to u$ in this case. 
 
We now assume that $u_n\ne u$ for all $n\in \mathbb{N}$. Since $\{u_n\}$ contains a subsequence $\{u_{n_k}\}$ such that $u_{n_k}\to u$ as $k\to \infty$, $u$ is a cluster point of $\{u_n\}$. Let $u'$ be another cluster point of $\{u_n\}$. So there exists a subsequence of $\{u_n\}$ which converges to $u'$. Then by a similar argument we can show that $u'$ is fixed point of $T$ which contradicts the uniqueness of $u$. Hence $u$ is the only cluster point of $\{u_n\}$. 

Let us now define $t_n=\rho(u_n, u)$ for all $n\in \mathbb{N}$. Then,
\begin{align*}
0\le t_{n+1}&=\rho(u_{n+1}, u)\\
&=\rho(Tu_n, Tu)\\
&<\max\{\rho(u_n, u), \rho(u_n, Tu_n), \rho(u, Tu)\}\\
&=\max\{\rho(u_{n}, u), \rho(u_n, u_{n+1})\}.
\end{align*}

If $0\le t_{n+1}<\rho(u_n, u)=t_n$ for all $n$, then $t_n\to \beta$ as $n \to \infty$ for some $\beta\geq 0$. Then
\begin{align*}
\beta&=\lim_n \rho(u_{n+1}, u)\\
&=\lim_k \rho(u_{n_k+1}, u)\\
&=\rho(Tu, u)=0
\end{align*}
which shows that $t_n\to 0$ as $n\to \infty$. Therefore, in this case, $u_n\to u$ as $n\to \infty$.

If $0\le t_{n+1}<\rho(u_n, u_{n+1})$, then $\displaystyle\mathop{0\le \lim_{n} t_{n+1}\le \lim_n \rho(u_n, u_{n+1})=0}$. This shows that $u_n\to u$ in this case also. Hence in either case, we see that $u_n\to u$ as $n\to \infty$ and the proof is complete.
\end{proof}
We now consider the following example.
\begin{example}\label{e6}
Let $X=\{2^n,3^n:n\in \mathbb{N}\}$. Let us define $\rho:X\times X\to \mathbb{R}$ by
$$ \rho(x, y)=\left\{%
\begin{array}{ll}
    0 &\hbox{if $x=y$}\\
    \frac{1}{x} & \hbox{if $x\neq 0, y=0$}\\
    \frac{1}{y} & \hbox{if $x=0, y\neq 0$}\\
    \frac{1}{x}+\frac{1}{y} & \hbox{if $x=2^n, y=3^m$ or $x=3^n, y=2^m$}\\
    1 & \hbox{if $x=2^n, y=2^m$ or $x=3^n, y=3^m$.}
\end{array}%
\right.$$

Then $(X, \rho)$ is a $b_4(2)$-metric space. It can also be verified that $(X, \rho)$ is sequentially compact. 

Define $T:X\to X$ by
$$ Tx=\left\{%
\begin{array}{ll}
    0 &\hbox{if $x$ is even}\\
    2 & \hbox{if $x$ is odd.}
\end{array}%
\right.$$
Then $T$ is not contractive since $\rho(T0, T3)=\rho(0, 2)=\frac{1}{2}\nless \frac{1}{3}=\rho(0, 3)$. Also $T$ satisfies 
$$\rho(Tx, Ty)<\max\{\rho(x, y), \rho(x, Tx), \rho(y, Ty)\}$$
for all $x, y\in X$ with $x\neq y$. Note that $0$ is the unique fixed point of $T$.
\end{example}
Next, we have the following theorem in the context of complete $b_v(s)$-metric spaces.
\begin{theorem}\label{t7}
Let $(X, \rho)$ be a complete $b_v(s)$-metric  space. Let $T:X\to X$ be a mapping satisfying the condition of Theorem~\ref{t5}. Furthermore, assume that for any $x\in X$ and  for any  $\varepsilon>0$, there exists a $\delta>0$ and an $N\in \mathbb{N}$ such that for $n,m\in \mathbb{N}$ with $n,m \geq N,$
$$\rho(T^nx, T^mx)<s^2\varepsilon+\delta\Longrightarrow \rho(T^{n+1}x, T^{m+1}x)\le \varepsilon.$$
 Then all the conclusions of Theorem \ref{t1} hold good. 
\end{theorem}
\begin{proof}
Let $u_0\in X$ be arbitrary but fixed and  consider the sequence $\{u_n\}$ where $u_n=T^nu_0$ for all $n\in \mathbb{N}$. If $u_n=u_{n+1}$ for some $n\in \mathbb{N}$, then the result is obvious. 

We set $s_n=\rho(u_n, u_{n+1})$ for all $n\in \mathbb{N}$. We claim that the sequence $\{s_n\}$ converges to $0$. Exactly in the same way as in Theorem~\ref{t5}, we see that the sequence $\{s_n\}$ is strictly decreasing. 

Since $s_n\ge 0$ for all $n$, it follows that $s_n\to \alpha$ for some $\alpha\ge 0$. If $\alpha>0$, then by given condition there is a $\delta'>0$  and an $N_1\in \mathbb{N}$ such that 
$$\rho(u_n, u_{n+1})<s^2\alpha+\delta'\Longrightarrow \rho(u_{n+1}, u_{n+2})\le \alpha$$
for all $n\geq N_1$.
 
By definition of $\alpha$, for this $\delta'>0$, there exists a sufficiently large $n\in \mathbb{N}$ such that
$$\rho(u_n, u_{n+1})<\alpha+\delta'\leq s^2\alpha+\delta'.$$
Therefore, $\rho(u_{n+1}, u_{n+2})\le \alpha$ and this leads to a contradiction. Hence we must have $\alpha=0$ i.e., $\displaystyle\mathop{\lim_{n\to \infty} \rho(u_n, u_{n+1})=0}$.

Next, we show that $\{u_n\}$ is a Cauchy sequence. Let $\varepsilon>0$ be arbitrary. Then there exists a $\delta>0$ and an $N_2\in \mathbb{N}$ such that 
$$\rho(T^nx, T^mx)<s^2\varepsilon+\delta\Longrightarrow \rho(T^{n+1}x, T^{m+1}x)\le \varepsilon$$
for all $n,m\geq N_2$.

Without loss of generality, we can assume that $\delta\le \varepsilon$. Since $\displaystyle\mathop{\lim_{n\to \infty} \rho(u_n, u_{n+1})=0}$, there exists an $N_3\in \mathbb{N}$ such that
$$\rho(u_n, u_{n+1})<\frac{\delta}{2(v+1)s^2}$$
for all $n\ge N_3$.

Let $n\in \mathbb{N}$ with $n\ge \max\{N_2,N_3\}+1$ be arbitrary. We now show by method of induction that
$$\rho(u_n, u_{n+k})\le \varepsilon$$
for all $k\in \mathbb{N}$.

Clearly, the result is true for $k=1$. Let the result be true for $k=1,2,\cdots, m$.

Case I: Let us first assume that $m>v$. Then
\begin{align*}
&\rho(u_{n-1}, u_{n+m})\\
&\le s\{\rho(u_{n-1}, u_n)+\rho(u_n, u_{n+1})+\cdots+\rho(u_{n+v-2}, u_{n+v-1})+\rho(u_{n+v-1}, u_{n+m})\}\\
&<s\Big\{\rho(u_{n-1}, u_n)+\rho(u_n, u_{n+1})+\cdots+\rho(u_{n+v-2}, u_{n+v-1})+\max\{\rho(u_{n+v-2}, u_{n+m-1}),\\
& \rho(u_{n+v-2}, u_{n+v-1}), \rho(u_{n+m-1}, u_{n+m})\}\Big\}.\hspace{20em} (1)
\end{align*} 
If $\max\{\rho(u_{n+v-2}, u_{n+m-1}), \rho(u_{n+v-2}, u_{n+v-1}), \rho(u_{n+m-1}, u_{n+m})\}=\rho(u_{n+v-2}, u_{n+v-1})$ or $=\rho(u_{n+m-1}, u_{n+m})$, then from $(1)$ we get
\begin{align*}
\rho(u_{n-1}, u_{n+m})&<s\left\{\frac{\delta}{2(v+1)s^2}+\frac{\delta}{2(v+1)s^2}+\cdots+\frac{\delta}{2(v+1)s^2}\right\}\\
&=\frac{\delta}{2s}<s\varepsilon+\frac{\delta}{2s}.
\end{align*}
If $\max\{\rho(u_{n+v-2}, u_{n+m-1}), \rho(u_{n+v-2}, u_{n+v-1}), \rho(u_{n+m-1}, u_{n+m})\}=\rho(u_{n+v-2}, u_{n+m-1})$, then from $(1)$, we get
\begin{align*}
&\rho(u_{n-1}, u_{n+m})\\
&\le s\{\rho(u_{n-1}, u_n)+\rho(u_n, u_{n+1})+\cdots+\rho(u_{n+v-2}, u_{n+v-1})+\rho(u_{n+v-2}, u_{n+m-1})\}\\
&<s\Big\{\rho(u_{n-1}, u_n)+\rho(u_n, u_{n+1})+\cdots+\rho(u_{n+v-2}, u_{n+v-1})+\max\{\rho(u_{n+v-3}, u_{n+m-2}),\\
& \rho(u_{n+v-3}, u_{n+v-2}), \rho(u_{n+m-2}, u_{n+m-2})\}\Big\}.\hspace{18em} (2) 
\end{align*}
If $\max\{\rho(u_{n+v-3}, u_{n+m-2}), \rho(u_{n+v-3}, u_{n+v-2}), \rho(u_{n+m-2}, u_{n+m-1})\}=\rho(u_{n+v-3}, u_{n+v-2})$ or $=\rho(u_{n+m-2}, u_{n+m-1})$, then from $(2)$ we get
\begin{align*}
\rho(u_{n-1}, u_{n+m})&<s\left\{\frac{\delta}{2(v+1)s^2}+\frac{\delta}{2(v+1)s^2}+\cdots+\frac{\delta}{2(v+1)s^2}\right\}\\
&=\frac{\delta}{2s}<s\varepsilon+\frac{\delta}{2s}.
\end{align*}

If $\max\{\rho(u_{n+v-3}, u_{n+m-2}), \rho(u_{n+v-3}, u_{n+v-2}), \rho(u_{n+m-2}, u_{n+m-1})\}=\rho(u_{n+v-3}, u_{n+m-2})$, then from $(2)$, we get
$$\rho(u_{n-1}, u_{n+m})<s\{\rho(u_{n-1}, u_n)+\rho(u_n, u_{n+1})+\cdots+\rho(u_{n+v-2}, u_{n+v-1})+\rho(u_{n+v-3}, u_{n+m-2})\}.$$
Continuing as above, we can either get
$$\rho(u_{n-1}, u_{n+m})<s\varepsilon+\frac{\delta}{2s}$$
or
\begin{align*}
&\rho(u_{n-1}, u_{n+m})\\
&\le s\{\rho(u_{n-1}, u_n)+\rho(u_n, u_{n+1})+\cdots+\rho(u_{n+v-2}, u_{n+v-1})+\rho(u_{n}, u_{n+m-v})\}\\
&<s\left\{\frac{\delta}{2(v+1)s^2}+\frac{\delta}{2(v+1)s^2}+\cdots+\frac{\delta}{2(v+1)s^2}+\varepsilon\right\}\\
&<\frac{\delta}{2s}+s\varepsilon.
\end{align*}

Case II: We now assume that $m<v$. Then
\begin{align*}
&\rho(u_{n-1}, u_{n+m})\\
&<s\{\rho(u_{n+m}, u_{n+m+1})+\rho(u_{n+m+1}, u_{n+m+2})+\cdots+\rho(u_{n+m+v-1}, u_{n+m+v})\\&+\rho(u_{n+m+v}, u_{n-1})\}.
\end{align*}
By Case I, we can conclude that
$$\rho(u_{n+m+v}, u_{n-1})<s\varepsilon+\frac{\delta}{2s}.$$
Therefore, we get
$$\rho(u_{n-1}, u_{n+m})<s\left\{\frac{\delta}{2(v+1)s^2}+\frac{\delta}{2(v+1)s^2}+\cdots+\frac{\delta}{2(v+1)s^2}+s\varepsilon+\frac{\delta}{2s}\right\}<\delta+s^2\varepsilon.$$

Case III: Let us finally consider $m=v$. In this case
\begin{align*}
&\rho(u_{n-1}, u_{n+m})\\
&\le s\{\rho(u_{n-1}, u_n)+\rho(u_n, u_{n+1})+\cdots+\rho(u_{n+v-2}, u_{n+v-1})+\rho(u_{n+v-1}, u_{n+v})\}\\
&<s\left\{\frac{\delta}{2(v+1)s^2}+\frac{\delta}{2(v+1)s^2}+\cdots+\frac{\delta}{2(v+1)s^2}\right\}\\
&=\frac{\delta}{2s}<s\varepsilon+\frac{\delta}{2s}.
\end{align*}
Thus, by combining all three cases, we find that
$$\rho(u_{n-1}, u_{n+m})<s^2\varepsilon+\delta.$$
Then by hypothesis, we get
$$\rho(u_n, u_{n+m+1})\le \varepsilon$$
which shows that the result is true for $k=m+1$. Therefore, by method of induction, we get
$$\rho(u_n, u_{n+k})\le \varepsilon$$
for all $n\ge \max\{N_2,N_3\}+1$ and for all $k\in \mathbb{N}$. Hence $\{u_n\}$ is a Cauchy sequence in $X$ and by completeness of $X$, we find an element $u\in X$ such that $u_n\to u$ as $n\to \infty$. That $u$ is the unique fixed point of $T$ and the sequence $\{T^nu_0\}$ converges to $u$ follows along the same line of proof of Theorem~\ref{t5}. 
\end{proof}
The following example will show that the additional condition assumed to prove the existence of fixed point of a mapping in the setting of a complete $b_v(s)$-metric space cannot be removed:

\begin{example}\label{e8}
Let $X=[0, \infty)$. Define $\rho:X\times X\to \mathbb{R}$ by
$$ \rho(x, y)=\left\{%
\begin{array}{ll}
    0 &\hbox{if $x=y$}\\
    1+2x+2y & \hbox{if $x>0, y>0$}\\
    x &\hbox{if $x\neq 0, y=0$}\\
    y &\hbox{if $x=0, y\neq 0$.}
\end{array}%
\right.$$
\end{example}
Then $(X, \rho)$ is a $b_2(2)$-metric space which is complete but not sequentially compact. 

Define a mapping $T:X\to X$ by
$$ Tx=\left\{%
\begin{array}{ll}
    \frac{1}{2} &\hbox{if $x=0$}\\
    \frac{x}{2} & \hbox{if $x\neq 0$.}
\end{array}%
\right.$$
Then $T$ satisfies
$$\rho(Tx, Ty)<\max\{\rho(x, y), \rho(x, Tx), \rho(y, Ty)\}$$
for all $x, y\in X$ with $x\neq y$. Yet the mapping $T$ does not admit any fixed point in $X$.

Let us consider the following example, which will ratify the above result:
\begin{example}\label{e9}
Let $X=[0, 1]$. Then $(X, \rho)$ is a complete $b_1(1)$-metric space. Define $T:X\to X$ by
$$ Tx=\left\{%
\begin{array}{ll}
    \frac{x}{2} \hbox{ if $x\in [0, 4]$}\\
    -2x+10 & \hbox{if $x\in [4, 5]$.}
\end{array}%
\right.$$
Then $T$ is not contractive since $\rho(T4, T5)=\rho(2, 0)=2\nless 1=\rho(4, 5)$. However it is easy to verify that $T$ satisfies 
$$\rho(Tx, Ty)<\max\{\rho(x, y), \rho(x, Tx), \rho(y, Ty)\}$$
for all $x, y\in X$ with $x\ne y$. Note that $0$ is the unique fixed point of $T$.
\end{example}
Finally, from Theorem~\ref{t3} or Theorem~\ref{t7}, we have the following corollary, and by this corollary, we get the answer of the open problem \eqref{tt2}.
\begin{corollary}\label{c10}
Let $(X, \rho)$ be a complete $b_v(s)$-metric  space. Let $T:X\to X$ be a mapping such that
$$\rho(Tx, Ty)<\rho(x, y)$$
for all $x, y\in X$ with $x\neq y$. Furthermore, assume that for any $x\in X$ and  for any  $\varepsilon>0$, there exists a $\delta>0$ and an $N\in \mathbb{N}$ such that for $n,m\in \mathbb{N}$ with $n,m \geq N,$
$$\rho(T^nx, T^mx)<s^2\varepsilon+\delta\Longrightarrow \rho(T^{n+1}x, T^{m+1}x)\le \varepsilon.$$
 Then $T$ has a unique fixed point.
\end{corollary}
Finally, we have the following theorem concerning the completeness of a $b_v(s)$-metric space via the fixed point property of  certain types of contractive mappings.
\begin{theorem}\label{t11}
Let $(X, \rho)$ be a $b_v(s)$-metric space. Assume that every self mapping $T$ on $(X,\rho)$ satisfying the condition
$$\rho(Tx, Ty)<b\rho(x, Tx)+c\rho(y, Ty)$$
for all $x, y\in X$ with $x\neq y$, where $b, c\in \mathbb{R}^+$ with $b+c=1$, has a unique fixed point. Then $(X, \rho)$ is complete. 
\end{theorem}
\begin{proof}
Let on the contrary that $(X, \rho)$ be not complete. So we can find a Cauchy sequence $\{u_n\}$ in $X$ such that for no $x\in X$, $\{u_n\}$ converges to $x$. Without loss of generality, let $u_n\neq u_m$ for all $m, n\in \mathbb{N}$. Let $A$ be the range set of $\{u_n\}$ and for any $x\in X$, consider the set $D(x,A)=\inf\{\rho(x,a):a\in A\}$. Then for any $x\notin A$, we have $D(x, A)>0$.

If $x\in A$, then there exists an $n_0\in \mathbb{N}$ such that $x=u_{n_0}$. We can then find an $n'_0\in \mathbb{N}$ such that 
$$\rho(u_m, u_{n'_0})<b\rho(u_{n_0}, u_{n'_0})\eqno(1)$$
for all $m\ge n'_0>n_0$. 

Again if $x\notin A$, then there exists an $n_x\in \mathbb{N}$ such that
\begin{align*}
\rho(u_m, u_{n_x})&<bD(x, A)\mbox{ for all }m\ge n_x\\
&\le b\rho(x, u_n) \mbox{ for all } n\in \mathbb{N}
\end{align*}
which implies that
$$\rho(u_m, u_{n_x})<b\rho(x, u_n)\mbox{ for all } m\ge n_x\mbox{ and for all }n\in \mathbb{N}.\eqno(2)$$
We now define a map $T:X\to X$ by
$$ Tx=\left\{%
\begin{array}{ll}
    u_{n'_0}, &\hbox{if $x\in A$ and}~x=u_{n_0};\\
    u_{n_x}, & \hbox{if $x\notin A$}.
\end{array}%
\right.$$

We claim that $T$ satisfies the condition
$$\rho(Tx, Ty)<b\rho(x, Tx)+c\rho(y, Ty)$$
for all $x, y\in X$ with $x\neq y$. 

For, let $x, y\in X$ with $x\neq y$. If $x, y\in A$, then there exists $n_1, n_2\in \mathbb{N}$ such that $x=u_{n_1}, y=u_{n_2}$. Therefore, $Tx=u_{n'_1}$, $Ty=u_{n'_2}$. Let us suppose that $n'_2\ge n'_1$. Then from $(1)$, we have
$$\rho(Tx, Ty)=\rho(u_{n'_2}, u_{n'_1})<b\rho(u_{n_1}, u_{n'_1})=b\rho(x, Tx)<b\rho(x, Tx)+c\rho(y, Ty).$$
Also, if $x, y\notin A$, then $Tx=u_{n_x}, Ty=u_{n_y}$ for some $n_x, n_y\in \mathbb{N}$. Take $n_y\ge n_x$. Then
$$\rho(Tx, Ty)=\rho(u_{n_x}, u_{n_y})<b\rho(x, u_{n_x})=b\rho(x, Tx)<b\rho(x, Tx)+c\rho(y, Ty).$$
Finally, if $x\notin A$ and $ y\in A$, then $y=u_{n_0}$ for some $n_0\in \mathbb{N}$. Then $Tx=u_{n_x}, Ty=u_{n'_0}$ for some $n_x\in \mathbb{N}$. If $n'_0\ge n_x$, then from $(2)$, we get
$$\rho(Tx, Ty)=\rho(u_{n'_0}, u_{n_x})<b\rho(x, u_{n_x})=b\rho(x, Tx)<b\rho(x, Tx)+c\rho(y, Ty)$$
and if $n_x\ge n'_0$, then from $(1)$, we get
$$\rho(Tx, Ty)=\rho(u_{n_x}, u_{n'_0})<b\rho(u_{n'_0}, u_{n_0})=b\rho(y, Ty)<b\rho(x, Tx)+c\rho(y, Ty).$$
Combining all the above considerations, we get
$$\rho(Tx, Ty)<b\rho(x, Tx)+c\rho(y, Ty)$$
for all $x, y\in X$ with $x\neq y$. It is important to note that $T$ admits no fixed point in $X$. This contradicts our hypothesis and hence $(X,\rho)$ is complete.
\end{proof}
%
%
%

\paragraph{\textbf{Acknowledgement}.}
The second named author would like to express his special thanks of gratitude to  CSIR, New Delhi, India, for their financial supports under the CSIR-SRF fellowship scheme (Award Number: $09/973(0018)/2017$-EMR-I). 

\begin{thebibliography}{10}

\bibitem{AK18}
M.S. Abdullahi and P.~Kumam.
\newblock Partial $b_v(s)$-metric spaces and fixed point theorems.
\newblock {\em J. Fixed Point Theory Appl.}, 20:1--13, 2018.

\bibitem{AMR18}
S.~Aleksi\'{c}, Z.D. Mitrovi\'{c}, and S.~Radenovi\'{c}.
\newblock A fixed point theorem of {J}ungck in $b_v(s)$-metric spaces.
\newblock {\em Period Math Hung}, 77:224--231, 2018.

\bibitem{B22}
S.~Banach.
\newblock Sur les op\'{e}rations dans les ensembles abstraits et leur
  application aux \'{e}quations int\'{e}grales.
\newblock {\em Fund. Math.}, 3:133--181, 1922.
\newblock JFM 48.1204.02.

\bibitem{DKMR20}
T.~Do$\check{s}$enovi\'{c}, Z.~Kadelburg, Z.D. Mitrovi\'{c}, and
  S.~Radenovi\'{c}.
\newblock New fixed point results in $b_v(s)$-metric spaces.
\newblock {\em Math. Slovaca}, 70(2):441--–452, 2020.

\bibitem{E1}
M.~Edelstein.
\newblock On fixed and periodic points under contractive mappings.
\newblock {\em J. Lond. Math. Soc.}, 37(1):74--79, 1962.

\bibitem{F1}
B.~Fisher.
\newblock A fixed point theorem for compact metric spaces.
\newblock {\em Publ. Math. Debrecen}, 25:193--194, 1978.

\bibitem{GDC}
H.~Garai, L.K. Dey, and Y.J. Cho.
\newblock On  contractive mappings and discontinuity at fixed points.
\newblock {\em Appl. Anal. Discrete Math.}, 14(1):033--054, 2020.

\bibitem{GDM}
H.~Garai, L.K. Dey and P.~Mondal.
\newblock The contractive principle for mappings in $b_v(s)$-metric spaces.
\newblock arXiv:1802.03136.

\bibitem{GDS}
H.~Garai, L.K. Dey, and T.~Senapati.
\newblock On {K}annan-type contractive mappings.
\newblock {\em Numer. Funct. Anal. Optim.}, 39(13):1466--1476, 2018.

\bibitem{G}
M.~Geraghty.
\newblock On contractive mappings.
\newblock {\em Proc. Amer. Math. Soc.}, 40(2):604--608, 1973.

\bibitem{G1}
J.~G\'{o}rnicki.
\newblock Fixed point theorems for {K}annan type mappings.
\newblock {\em J. Fixed Point Theory Appl.}, 19(3):2145--2152, 2017.

\bibitem{MAKR19}
Z.D. Mitrovi\'{c}, H.~Aydi, Z.~Kadelburg, and G.~Soleimani Rad.
\newblock On some rational contractions in $b_v(s)$-metric spaces.
\newblock {\em Circ. Mat. Palermo, II. Ser}, 2019.
\newblock https://doi.org/10.1007/s12215-019-00465-6.

\bibitem{MR}
Z.D. Mitrovi\'{c} and S.~Radenovi\'{c}.
\newblock The {B}anach and {R}eich contractions in $b_v(s)$-metric spaces.
\newblock {\em J. Fixed Point Theory Appl.}, 19(4):3087--3095, 2017.

\bibitem{S2}
T. Suzuki.
\newblock The weakest contractive conditions for {E}delstein’s mappings to have a fixed point in complete metric spaces.
\newblock {\em J. Fixed Point Theory Appl.}, 19(4):2361--2368, 2017.

\end{thebibliography}

\end{document}